\documentclass[11pt]{article}
\usepackage{amsmath,amssymb,amsthm}
\usepackage{algorithm,algorithmic}
\usepackage{graphicx}
\usepackage[colorlinks=true,linkcolor=blue,citecolor=blue,urlcolor=blue]{hyperref}
\usepackage{color}
\usepackage{enumerate}
\usepackage{placeins}
\usepackage{geometry}
\newtheorem{theorem}{Theorem}
\newtheorem{lemma}{Lemma}
\newtheorem{corollary}{Corollary}

\newtheorem{remark}{Remark}
\newtheorem{proposition}{Proposition}

\title{Update-Magnitude State Redistribution (UM-SRD):\\
A Shut-off Extension of Weighted SRD for Cut-Cell Methods}
\author{Justo E.\ Karell \\ Department of Computer Science \\ Stevens Institute of Technology \\ Hoboken, NJ, USA \\
\texttt{justokarell@gmail.com}
}
\date{\today}

\begin{document}
\maketitle

\begin{abstract}
We introduce Update-Magnitude State Redistribution (UM-SRD), a shut-off
extension of the weighted state redistribution (SRD) method of Berger
\& Giuliani for cut-cell finite-volume schemes. Standard weighted SRD
continues to redistribute states even when the finite-volume update
vanishes. UM-SRD modifies the redistribution step by blending the SRD
operator with the identity through a local update-dependent parameter,
reducing exactly to the base scheme when updates vanish. For a
one-dimensional linear advection model problem with a single small cut
cell, we prove conservation and total variation diminishing (TVD)
properties under the same CFL condition as the base upwind scheme, and
analyze the perturbation introduced by the blending operator. We show
that the modification preserves the first-order consistency of the base
scheme and exactly preserves steady states when finite-volume updates are
zero. Numerical experiments in one and two dimensions confirm first-order 
convergence, exact shut-off behavior, stabilization near small cut cells 
under the full-mesh timestep, and exact steady-state preservation. The 
shut-off mechanism is active when finite-volume updates vanish or nearly 
vanish; in fully active advection UM-SRD reduces to standard SRD by design.
\end{abstract}



\section{Introduction}
Simulating fluid flow, heat transfer, or wave propagation around complex 
geometries --- aircraft wings, vascular networks, turbine blades --- requires 
solving partial differential equations on domains with irregular boundaries. 
The standard approach is to generate a mesh that conforms to the boundary, 
but constructing such meshes for three-dimensional geometries is 
time-consuming, often requiring significant manual intervention, and must be 
repeated whenever the geometry changes. Cut-cell methods offer an alternative: 
discretize on a simple background Cartesian grid and represent the geometry 
by cutting the grid cells that intersect the boundary. Cells that intersect 
the boundary are retained as irregular \emph{cut cells} rather than discarded, 
allowing complex geometries to be handled without generating a body-fitted 
mesh. The bulk of the computation remains on a uniform Cartesian structure, 
and grid generation reduces to a geometry intersection problem that can be 
automated. The cost is that cells near the boundary can have arbitrarily small 
volume fractions. Under standard explicit finite-volume schemes, these small 
cells impose a time-step restriction proportional to their volume, which can 
be orders of magnitude smaller than the timestep required for stability on 
the background Cartesian grid~\cite{colella1998,quirk1994}.

Several approaches have been developed to avoid the small-cell timestep 
restriction, including cell merging, h-box methods, and flux 
redistribution~\cite{colella2006,berger1989}. Berger \& Giuliani~\cite{berger2024} 
introduced weighted state redistribution (SRD) as a provably stable 
alternative. Their method constructs merging neighborhoods around small cut 
cells, evolves neighborhood-averaged states on a merged mesh, and 
redistributes the result back to the original mesh using carefully chosen 
weights, retaining conservation and TVD properties without restricting the 
timestep to the small-cell size. They prove weighted SRD is TVD for 
one-dimensional scalar advection and demonstrate robust stabilization near 
cut cells with volume fraction as small as $\alpha = 0.05$. They explicitly 
note, however, that SRD has no mechanism to reduce redistribution when the 
finite-volume update is small: the operator fires at every time step 
regardless of update magnitude, even near steady states or under very small 
timesteps. They describe this as an open gap relative to classical flux 
redistribution methods~\cite{colella2006,berger1989}, which naturally vanish 
when the flux imbalance is negligible.

This paper introduces Update-Magnitude State Redistribution (UM-SRD), a 
direct extension of weighted SRD that closes this gap. The absence of a 
shut-off mechanism has two concrete consequences. First, SRD does not 
preserve steady states of the underlying finite-volume scheme: even when 
the flux balance is exactly zero, SRD continues to redistribute state values 
across small-cell neighborhoods, replacing pointwise values with their 
volume-weighted averages (see Proposition~\ref{prop:srd-steady}). Second, in 
smooth regions where the finite-volume update is small but nonzero, repeated 
redistribution introduces numerical dissipation proportional to the gradient 
across the neighborhood, independent of the actual dynamics. UM-SRD eliminates both effects by blending the SRD operator with the 
identity through a local update-dependent parameter that reduces exactly 
to zero when the finite-volume update vanishes. In fully active advection, 
where updates are nonzero at every step, the blending parameter satisfies 
$s_j \approx 1$ and UM-SRD reduces to standard SRD; the two methods are 
identical in this regime by design. The benefit of UM-SRD is concentrated 
in near-steady or quasi-steady regimes where SRD continues to apply 
redistribution unnecessarily.

\medskip

\noindent\textbf{Analytical results.}
For a one-dimensional linear advection model problem with a single small
cut cell, we establish two properties of UM-SRD. First, the UM-SRD
time-stepping operator is TVD on the original cut-cell mesh under the same
CFL condition as the base first-order upwind scheme, by virtue of the TVD
structure of weighted SRD and the convex-combination form of the blending.
Second, UM-SRD may be viewed as a perturbation of the base scheme whose
size is proportional to the update-dependent blending parameter; when this
parameter scales with the update magnitude, the contribution to the local
truncation error is of higher order than that of the base scheme. These
results are confirmed numerically in Section~\ref{sec:numerics}.

\medskip
UM-SRD is deliberately minimal: it reuses the weighted SRD neighborhood and
weight structure unchanged, adding only a local, smoothly varying blending
parameter per neighborhood. We give a complete algorithmic description in
Section~\ref{sec:algorithm} and numerical experiments in one and two
dimensions in Section~\ref{sec:numerics}, illustrating shut-off behavior,
stabilization near small cut cells, and steady-state preservation.
\paragraph{Notation and standing assumptions.}
We approximate the solution $u(x,t)$ to the hyperbolic conservation law
\[
  u_t + \nabla \cdot f(u) = 0, \qquad x \in \Omega \subset \mathbb{R}^d,
  \quad t > 0,
\]
where $\Omega$ is a bounded domain with piecewise smooth boundary,
$d \in \{1, 2, 3\}$ is the spatial dimension, and $f(u)$ is a smooth flux
function. The domain is discretized into a finite collection of
non-overlapping regions called \emph{cells}, indexed by $i$. Each cell
$I_i \subset \Omega$ is a control volume with volume $V_i = |I_i|$; in one
space dimension, $V_i = \Delta x_i$. Cells intersected by the domain
boundary can have arbitrarily small volume; a cell is called a \emph{small
cut cell} if $V_i = \alpha_i h$ with $0 < \alpha_i < 1$, where $h$ denotes
the background grid cell width.

The primary unknowns are cell averages
\[
  U_i^n \approx \frac{1}{V_i}\int_{I_i} u(x,t_n)\,dx, \qquad t_n = n\Delta t,
\]
approximating the mean value of $u$ over cell $I_i$ at time $t_n$. After
one step of the base finite-volume scheme and before redistribution, the
updated averages are denoted $U^*$. Small cut cells are grouped into
disjoint merging neighborhoods $M_j$, where $j$ indexes the neighborhood
and $M_j$ is the set of cell indices it contains. For each cell $i$, $W_i$
denotes the set of neighborhood indices $j$ such that cell $i$ receives a
contribution from $M_j$ during redistribution. The weighted neighborhood
volume is $\hat{V}_j = \sum_{i \in M_j} w_{i,j} V_i$, where $w_{i,j} \ge
0$ are the SRD weights satisfying $\sum_{j \in W_i} w_{i,j} = 1$ for all
$i$. We write $B$ for one application of the base finite-volume scheme,
$S_j$ for the SRD operator on neighborhood $M_j$, $\mathrm{Id}$ for the
identity operator, and $TV(U) = \sum_{i=0}^{N-1} |U_i - U_{i-1}|$ for the
total variation on the original grid. We assume the CFL condition $0 \le
a\,\Delta t / \Delta x_i \le 1$ for all cells $i$.

The TVD and accuracy results in Section~\ref{sec:analysis-1D} are proved
for scalar linear advection $u_t + au_x = 0$ ($a > 0$) with one small cut
cell of volume fraction $\alpha < 1/2$, merging left, first-order upwind
flux $F(U_i, U_{i+1}) = aU_i$, and CFL condition $a\Delta t/h \le 1$.
\section{Related Work}

Flux redistribution methods~\cite{colella2006,berger1989} address the 
small-cell timestep restriction by redistributing mass or fluxes from small 
cells to neighboring cells conservatively and stably. A key property is that 
the post-processing vanishes when the finite-volume update is small: if the 
flux imbalance is negligible, the redistribution step has no effect and the 
base scheme is recovered exactly.

Berger \& Giuliani~\cite{berger2024} introduced weighted SRD as a provably 
stable alternative. Their method constructs merging neighborhoods around 
small cut cells, evolves neighborhood-averaged states on a merged mesh, and 
redistributes the result back using carefully chosen weights. They prove 
weighted SRD is TVD for one-dimensional scalar advection and demonstrate 
robust stabilization near arbitrarily small cut cells. The method does not 
have a vanishing-update property: the redistribution operator is applied at 
every time step regardless of update magnitude, in contrast to flux 
redistribution whose post-processing diminishes automatically as the update 
tends to zero.

The approach of applying corrections only where the solution requires them 
appears in the Multidimensional Optimal Order Detection (MOOD) framework of 
Clain, Diot, and Loub\`ere~\cite{clain2011}, in which a cell-by-cell 
a posteriori detection step determines whether a high-order update must be 
replaced by a more diffusive one to satisfy admissibility constraints. UM-SRD 
shares the same local, solution-dependent activation principle but operates 
as an a priori blending parameter on the redistribution operator rather than 
as an a posteriori limiter on the order of accuracy. The blending concept is also related to update-dependent limiting strategies 
in finite-volume methods~\cite{vanleer1979,sweby1984}, adapted here to the 
state redistribution framework. Other approaches to embedded-boundary 
stabilization include explicit-implicit hybrid schemes~\cite{may2017} and 
moving-boundary cut-cell formulations~\cite{schneiders2016}; UM-SRD is 
complementary to these in that it targets the redistribution post-processing 
step rather than the flux computation or time integration.

\section{Definition of the UM-SRD Method}\label{sec:framework}

Berger \& Giuliani~\cite{berger2024} define a time step in three stages. 
Starting from $U^n$, the base finite-volume scheme produces the intermediate 
state
\begin{equation}
  U^*_i = U^n_i - \frac{\Delta t}{V_i} \sum_{f \in \partial I_i} F_{i,f},
  \label{eq:fv-update}
\end{equation}
where $F_{i,f}$ is the numerical flux through face $f$ of cell $I_i$. For 
systems with bidirectional characteristics, a pre-merging stage precedes 
this update as described in \cite{berger2024}, Section~4.1. The 
redistribution step then maps $U^*$ to $U^{n+1}$ by first forming a 
volume-weighted average over each merging neighborhood $M_j$,
\begin{equation}
  \hat{Q}^*_j = \frac{1}{\hat{V}_j} \sum_{i \in M_j} w_{i,j} V_i U^*_i,
  \label{eq:Qhat}
\end{equation}
and then assigning a weighted combination of these averages back to each 
cell:
\begin{equation}
  [S_j U^*]_i = \sum_{k \in W_i} w_{i,k} \hat{Q}^*_k,
  \qquad i \in M_j.
  \label{eq:SRD-op}
\end{equation}
The weights satisfy $\sum_{j \in W_i} w_{i,j} = 1$ for all $i$, ensuring 
conservation, and $S_j$ acts as the identity outside $M_j$. As noted in 
\cite{berger2024}, this operator fires at every time step regardless of 
the magnitude of $U^* - U^n$: when the finite-volume update is small or 
zero, $S_j$ continues to replace cell values with neighborhood averages, 
modifying the solution even where no stabilization is needed.

UM-SRD addresses this by introducing a scalar blending parameter $s_j \in 
[0,1]$ that measures whether redistribution is warranted in neighborhood 
$M_j$. Define the maximum update magnitude over the neighborhood,
\begin{equation}
  \Delta U_{\max,j} = \max_{i \in M_j} \|U^*_i - U^n_i\|_2,
  \label{eq:DUmax}
\end{equation}
and the normalized indicator
\begin{equation}
  \eta_j = \frac{\Delta U_{\max,j}}{\varepsilon + \Delta U_{\max,j}},
  \label{eq:eta}
\end{equation}
where $\varepsilon > 0$ (e.g.\ $\varepsilon = 10^{-14}$) prevents division 
by zero. By construction $\eta_j = 0$ if and only if all 
updates in $M_j$ are at machine precision, and $\eta_j \approx 1$ whenever 
any cell has a nonzero update. The blending parameter is then
\begin{equation}
  s_j = \frac{\eta_j^p}{\eta_j^p + \tau^p}
  \label{eq:shutoff}
\end{equation}
for parameters $p \ge 1$ and $\tau > 0$. This function is strictly 
increasing in $\eta_j$, belongs to $C^\infty[0,1]$, satisfies $s_j = 0$ 
when $\eta_j = 0$, and satisfies $s_j \to 1$ as $\tau \to 0$.

The full SRD step $U^{n+1} = S_j U^*$ is replaced by the blended operator
\begin{equation}
  R_j(s_j) = (1 - s_j)\,\mathrm{Id} + s_j\,S_j,
  \label{eq:Rj}
\end{equation}
where $\mathrm{Id}$ denotes the identity. The global operator $R$ acts as 
$R_j$ on each small-cell neighborhood and as the identity elsewhere, giving 
in component form
\begin{equation}
  U^{n+1}_i = (1 - s_j)\,U^*_i + s_j \sum_{k \in W_i} w_{i,k} \hat{Q}^*_k,
  \qquad i \in M_j.
  \label{eq:UMSRD-update}
\end{equation}
When all updates in $M_j$ vanish, $s_j = 0$ and $R_j(0) = \mathrm{Id}$, 
so $U^{n+1} = U^*$ exactly. When updates are large, $s_j \approx 1$ and 
$R_j \approx S_j$, recovering standard SRD. The normalized indicator 
\eqref{eq:eta} behaves nearly as a binary switch: $s_j = 0$ exactly when 
all updates are at machine precision, and $s_j \approx 1$ as soon as any 
cell in $M_j$ has a nonzero update. This is a deliberate design choice. 
Exact shut-off at machine precision is what guarantees exact steady-state 
preservation (Proposition~\ref{prop:srd-steady}); a smoothly varying but 
always-positive $s_j$ would reduce but not eliminate the steady-state drift 
identified in Section~\ref{sec:steady-state}. The binary character of the 
normalized indicator is therefore a feature of the design, not a limitation.

\begin{proposition}[Conservation]
\label{prop:conservation}
For any $s_j \in [0,1]$, $R_j(s_j)$ satisfies
\[
  \sum_{i \in M_j} V_i\, U^{n+1}_i = \sum_{i \in M_j} V_i\, U^*_i.
\]
\end{proposition}

\begin{proof}
For cells outside $M_j$, $U^{n+1}_i = U^*_i$ trivially. For cells in 
$M_j$, assuming each cell belongs to exactly one neighborhood so that 
$W_i = \{j\}$, substituting \eqref{eq:UMSRD-update} gives
\begin{align*}
  \sum_{i \in M_j} V_i\,U^{n+1}_i
  &= (1-s_j)\sum_{i \in M_j} V_i U^*_i
     + s_j \sum_{i \in M_j} V_i \hat{Q}^*_j \\
  &= (1-s_j)\sum_{i \in M_j} V_i U^*_i
     + s_j\, \hat{Q}^*_j \hat{V}_j \\
  &= (1-s_j)\sum_{i \in M_j} V_i U^*_i
     + s_j \sum_{i \in M_j} V_i U^*_i
   = \sum_{i \in M_j} V_i U^*_i. \qedhere
\end{align*}
\end{proof}
\section{Analysis of UM-SRD for a 1D Model Problem}
\label{sec:analysis-1D}

We work throughout under the 1D model setting stated in 
Section~\ref{sec:framework}: scalar linear advection $u_t + au_x = 0$ 
($a > 0$) with one small cut cell of volume fraction $\alpha < 1/2$ merged 
to the left, first-order upwind flux $F(U_i, U_{i+1}) = aU_i$, and CFL 
condition $a\Delta t/h \le 1$. Total variation is $TV(U) = 
\sum_{i=0}^{N-1}|U_i - U_{i-1}| = \|DU\|_1$, where $D$ denotes the 
first-difference operator. One UM-SRD timestep is the operator composition
\[
U^n
\;\xrightarrow{\text{pre-merge}}\;
Q^{n,-1}
\;\xrightarrow{\text{FV update}}\;
Q^n
\;\xrightarrow{\text{redistribution}}\;
U^{n+1}
\;=\;
R \circ B \circ P(U^n),
\]
where $P$ denotes the pre-merging operator, $B$ the merged-mesh 
finite-volume update, and $R$ the blended redistribution operator.

\subsection{TVD Properties}

\begin{lemma}[SRD is TVD]
\label{lem:TVD-SRD}
With the one-small-cell merging-left configuration and monotone weights
$w_{-1,-1} = \alpha$, the SRD operator satisfies
\[
  TV\!\left(S_j(U^{\mathrm{init}})\right) \le TV(U^{\mathrm{init}})
\]
for all initial data $U^{\mathrm{init}}$.
\end{lemma}

\begin{proof}
Berger \& Giuliani~\cite{berger2024} establish that the pre-merging stage
is TVD and that redistribution back to the original mesh is TVD. The
composition of two TVD operators is TVD, so the result follows.
\end{proof}

\begin{lemma}[Finite-volume update is TVD]
\label{lem:TVD-Q-update}
Let $\{Q_i^n\}$ denote the merged-mesh cell averages updated by the
first-order upwind scheme. If $0 \le a\Delta t / \Delta x_i \le 1$ for
all merged cells, then $TV(Q^n) \le TV(Q^{n-1})$.
\end{lemma}

\begin{proof}
The upwind update reads
\[
  Q_i^n = (1 - a\lambda_i) Q_i^{n-1} + a\lambda_i Q_{i-1}^{n-1},
  \qquad \lambda_i = \Delta t / \Delta x_i.
\]
Since $0 \le a\lambda_i \le 1$, each updated value is a convex combination
of neighboring states, and the conclusion follows from Harten's
lemma~\cite{harten1983lax}.
\end{proof}

Because every merged cell satisfies $\Delta x_i \ge h$, the full-mesh CFL
condition $a\Delta t/h \le 1$ implies $a\Delta t/\Delta x_i \le 1$ on the
merged mesh, so the hypothesis of Lemma~\ref{lem:TVD-Q-update} is always
satisfied.

\begin{lemma}[Blended redistribution is TVD]
\label{lem:TVD-R}
For any $s_j \in [0,1]$,
\[
  TV\!\left(R_j(s_j) Q^n\right) \le TV(Q^n).
\]
\end{lemma}

\begin{proof}
Writing $R_j(s_j) = (1-s_j)\,\mathrm{Id} + s_j S_j$ and applying the 
triangle inequality gives
\[
  TV(U^{n+1})
  = \|D[(1-s_j)Q^n + s_j S_j(Q^n)]\|_1
  \le (1-s_j)\|DQ^n\|_1 + s_j\|DS_j(Q^n)\|_1.
\]
Lemma~\ref{lem:TVD-SRD} gives $\|DS_j(Q^n)\|_1 \le \|DQ^n\|_1$, so
\[
  TV(U^{n+1}) \le (1-s_j)\,TV(Q^n) + s_j\,TV(Q^n) = TV(Q^n). \qedhere
\]
\end{proof}

\begin{theorem}[UM-SRD is TVD]
\label{thm:TVD-UMSRD}
Under the 1D model assumptions, $TV(U^{n+1}) \le TV(U^n)$ for all time
steps.
\end{theorem}

\begin{proof}
Lemma~\ref{lem:TVD-SRD} gives $TV(Q^{n,-1}) \le TV(U^n)$ after 
pre-merging. Lemma~\ref{lem:TVD-Q-update} gives $TV(Q^n) \le 
TV(Q^{n,-1})$ after the finite-volume update on the merged mesh. 
Lemma~\ref{lem:TVD-R} gives $TV(U^{n+1}) \le TV(Q^n)$ after blended 
redistribution. Chaining the three inequalities yields $TV(U^{n+1}) \le 
TV(U^n)$.
\end{proof}

\subsection{Accuracy Analysis}
\label{sec:accuracy}

Using $R_j(s_j) = (1-s_j)\,\mathrm{Id} + s_j S_j$, the UM-SRD update 
satisfies
\begin{equation}
  U^{n+1,\mathrm{UM}} - U^{n+1,\mathrm{base}}
  = s_j^n (S_j - \mathrm{Id})\, U^{n+1,\mathrm{base}},
  \label{eq:perturbation}
\end{equation}
so UM-SRD acts as a perturbation of the base scheme whose size is controlled 
by $s_j^n$. The normalized indicator \eqref{eq:eta} satisfies 
$s_j^n = \mathcal{O}(1)$ whenever any update in $M_j$ is nonzero, and 
$s_j^n = 0$ exactly when all updates vanish. This is the implementation used 
in all experiments.

\begin{proposition}[First-order accuracy]
\label{prop:local-consistency}
With the normalized indicator of \eqref{eq:eta}, UM-SRD remains first-order 
accurate: for smooth solutions,
\[
  \|\tau^{n,\mathrm{UM}} - \tau^{n,\mathrm{base}}\|_\infty = \mathcal{O}(h).
\]
\end{proposition}

\begin{proof}
From \eqref{eq:perturbation}, $\tau^{n,\mathrm{UM}} - \tau^{n,\mathrm{base}} 
= s_j^n (S_j - \mathrm{Id}) B(u^n)$. In the one-small-cell merging-left 
configuration,
\begin{align*}
  [(S_j - \mathrm{Id})V]_{j-1} &= \alpha(V_j - V_{j-1}), \\
  [(S_j - \mathrm{Id})V]_j     &= (1-\alpha)(V_{j-1} - V_j),
\end{align*}
with $(S_j - \mathrm{Id})V = 0$ at all other indices. For smooth solutions 
$\|(S_j - \mathrm{Id})B(u^n)\|_\infty = \mathcal{O}(h)$, and since 
$s_j^n = \mathcal{O}(1)$,
\[
  \|\tau^{n,\mathrm{UM}} - \tau^{n,\mathrm{base}}\|_\infty
  = \mathcal{O}(1) \cdot \mathcal{O}(h) = \mathcal{O}(h), \qedhere
\]
matching the truncation error of the underlying first-order scheme.
\end{proof}

\begin{remark}[Higher-order perturbation with an unnormalized indicator]
\label{rem:unnormalized}
If the blending parameter is instead defined via an unnormalized indicator 
satisfying $|s_j^n| \le C_s h^p$ for some $p \ge 1$ with $C_s$ independent 
of $h$ and $n$, then \eqref{eq:perturbation} gives 
$\|\tau^{n,\mathrm{UM}} - \tau^{n,\mathrm{base}}\|_\infty = \mathcal{O}(h^{p+1})$,
a higher-order perturbation than the base scheme. The normalized indicator 
is preferred in practice because it achieves exact shut-off at machine 
precision; the unnormalized variant provides a higher-order accuracy guarantee 
in smooth regions at the cost of only approximate shut-off.
\end{remark}

\subsection{Steady-State Preservation}
\label{sec:steady-state}

\begin{proposition}[SRD perturbs non-constant steady states]
\label{prop:srd-steady}
Suppose $U_i^* = U_i^n$ for all $i$. If $U_{j-1}^n \ne U_j^n$, then 
standard SRD produces $U^{n+1} \ne U^n$.
\end{proposition}

\begin{proof}
The neighborhood average over the merged cell $\{j-1, j\}$ is
\[
  \hat{Q}^* = \frac{h U_{j-1}^n + \alpha h U_j^n}{h + \alpha h}
  = (1-\alpha) U_{j-1}^n + \alpha U_j^n.
\]
SRD assigns $\hat{Q}^*$ to both cells, so $U_{j-1}^{n+1} = U_j^{n+1} = 
\hat{Q}^* \ne U_{j-1}^n$ whenever $U_{j-1}^n \ne U_j^n$.
\end{proof}

UM-SRD does not share this defect. When the finite-volume update vanishes, 
$\eta_j = 0$, $s_j = 0$, and $R_j(0) = \mathrm{Id}$, so $U^{n+1} = U^n$ 
exactly.

\begin{corollary}[Steady-state error under repeated SRD]
\label{cor:steady-accumulation}
Let $\varepsilon_i = U_i^0 - u(x_i)$ denote the initial discretization 
error for a smooth steady state. Under repeated application of SRD with 
zero flux updates,
\[
  U_{j-1}^N - u(x_{j-1})
  = U_j^N - u(x_j)
  = (1-\alpha)\varepsilon_{j-1} + \alpha\varepsilon_j.
\]
\end{corollary}

\begin{proof}
With $U^* = U^n$, the neighborhood average at each step is 
$\hat{Q}^* = (1-\alpha)U_{j-1}^n + \alpha U_j^n$. After one SRD step, 
both cells satisfy $U_{j-1}^1 = U_j^1 = \hat{Q}^*$, and since the flux 
update remains zero, every subsequent step leaves this value unchanged. 
Subtracting the steady state $u(x_i)$ from both sides gives the result.
\end{proof}
\section{Numerical Experiments}\label{sec:numerics}

All experiments use $u_t + au_x = 0$ ($a > 0$) with one small cut cell of 
volume fraction $0 < \alpha < 1/2$ merged left and a first-order upwind base 
scheme. We compare the base upwind scheme, weighted SRD~\cite{berger2024}, 
and UM-SRD with default parameters $p=2$, $\tau=0.1$, $\varepsilon=10^{-14}$. 
UM-SRD adds $\mathcal{O}(|M_j|)$ operations per neighborhood per step; 
wall-clock overhead was less than 5\% relative to SRD in all 1D experiments.

\subsection{Experiment 1: Smooth Advection and Convergence}
\label{sec:exp-smooth}

On the domain $[0,1]$ with periodic boundary conditions, initial condition 
$u(x,0) = \sin(2\pi x)$, small cut cell at $x = 0.5$, and final time $T=1$, 
Table~\ref{tab:exp1_errors} reports $L^1$ and $L^\infty$ errors for UM-SRD 
at four grid levels. The base upwind scheme and weighted SRD produce nearly 
identical errors. Convergence rates approach first order ($0.86 \to 0.93 \to 
0.97$ in $L^1$), and UM-SRD and weighted SRD agree to within 0.01\%, 
confirming the blending does not degrade accuracy.

\begin{table}[ht]
\centering
\caption{$L^1$ and $L^\infty$ errors and convergence rates for UM-SRD,
Experiment~1. $N$ = number of full cells.}
\label{tab:exp1_errors}
\begin{tabular}{cccccc}
\hline
$N$ & $h$ & $E^{(1)}$ & $p^{(1)}$ & $E^{(\infty)}$ & $p^{(\infty)}$ \\
\hline
 40 & 0.0250 & 2.438e-01 & ---  & 3.839e-01 & --- \\
 80 & 0.0125 & 1.341e-01 & 0.86 & 2.088e-01 & 0.88 \\
160 & 0.0063 & 7.027e-02 & 0.93 & 1.092e-01 & 0.93 \\
320 & 0.0031 & 3.592e-02 & 0.97 & 5.609e-02 & 0.96 \\
\hline
\end{tabular}
\end{table}

\begin{figure}[htbp]
  \centering
  \includegraphics[width=0.6\textwidth]{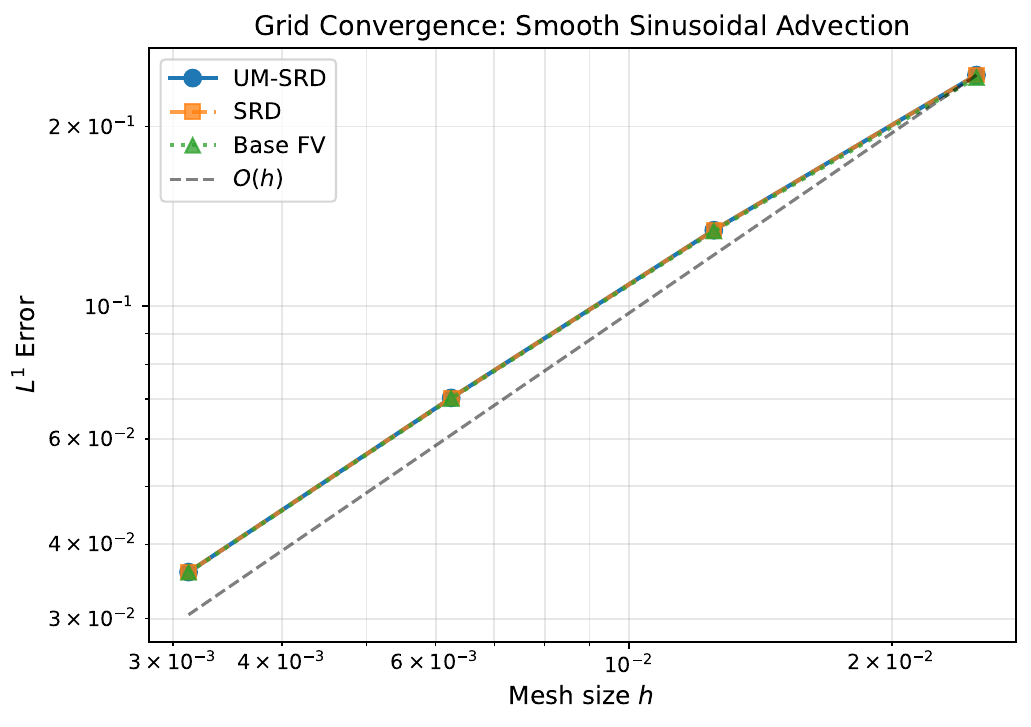}
  \caption{Experiment~1: $L^1$ convergence for UM-SRD, weighted SRD, and the
  base upwind scheme. All three are indistinguishable.}
  \label{fig:convergence}
\end{figure}

\FloatBarrier
\subsection{Experiment 2: Shut-off Behaviour}
\label{sec:exp-shutoff}

A compact cosine pulse with center $x_c = 0.10$ and half-width $w/2 = 0.04$ 
is advected past a small cell at $x = 0.5$ with $\alpha = 0.2$, $N = 200$ 
cells, and two CFL values (0.5 and 0.01). Figure~\ref{fig:sj-time} shows 
that $s_j^n = 0$ exactly before the pulse reaches the small cell ($t \lesssim 
0.14$), rises to $\approx 0.99$ while the pulse crosses, and returns to zero 
after the trailing edge exits ($t \gtrsim 0.85$). Both CFL cases produce 
identical $s_j^n$ histories because $\eta_j$ depends only on whether updates 
are nonzero, not on their magnitude.

\begin{figure}[htbp]
  \centering
  \includegraphics[width=0.85\textwidth]{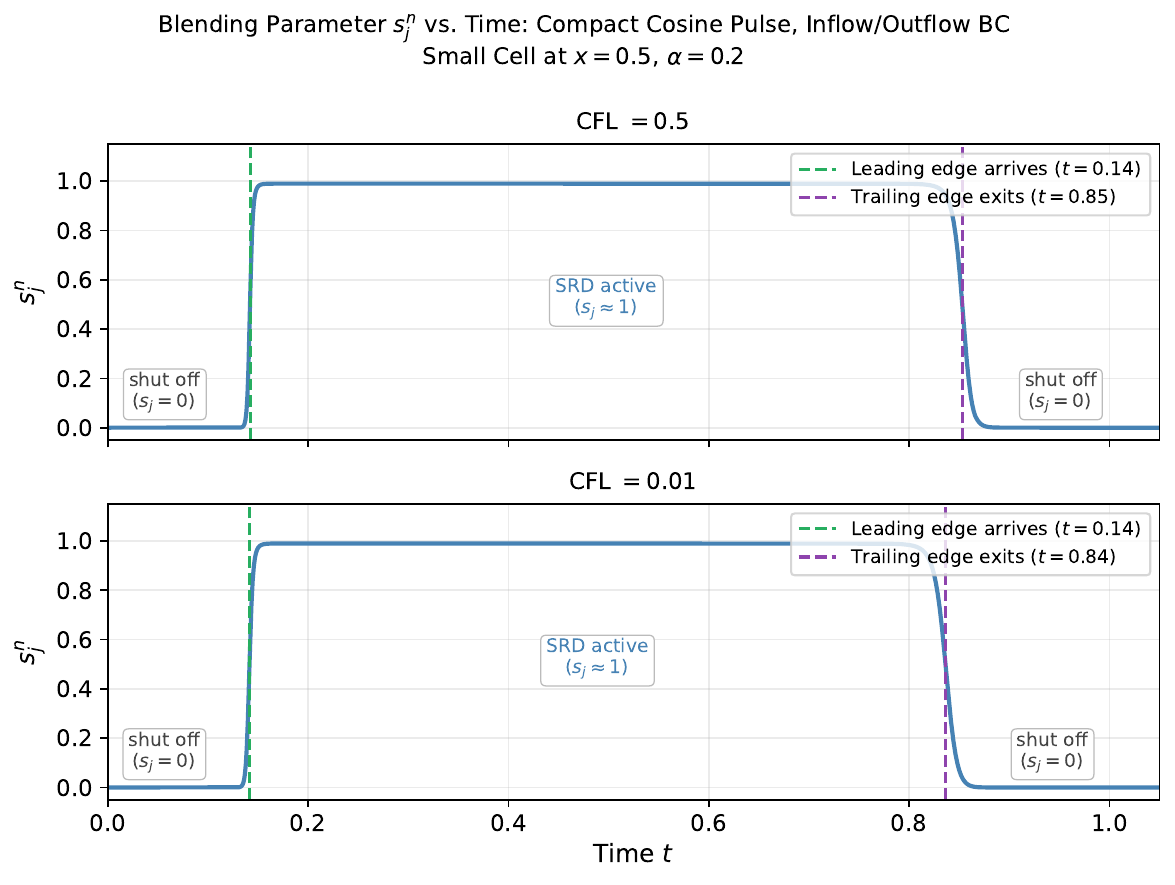}
  \caption{Experiment~2: time history of $s_j^n$ for CFL\,$=0.5$ (top) and
  CFL\,$=0.01$ (bottom). Both cases produce identical $s_j^n$ histories,
  confirming the binary nature of the normalized indicator. Update magnitudes
  differ by a factor of fifty, consistent with $|\Delta U| = \mathcal{O}(\Delta t)$.}
  \label{fig:sj-time}
\end{figure}

\FloatBarrier
\subsection{Experiment 3: Discontinuous Data and TVD Behaviour}
\label{sec:exp-step}

With step initial condition $u(x,0) = \mathbf{1}_{x < x_0}$ and the small 
cut cell near $x_0$, all three methods remain non-oscillatory throughout the 
simulation, consistent with Theorem~\ref{thm:TVD-UMSRD}. UM-SRD closely 
matches weighted SRD near the small cell ($s_j \approx 1$ while the 
discontinuity crosses) and introduces no new extrema. The profiles are shown 
in Figure~\ref{fig:step-profile}.

\begin{figure}[htbp]
  \centering
  \includegraphics[width=0.6\textwidth]{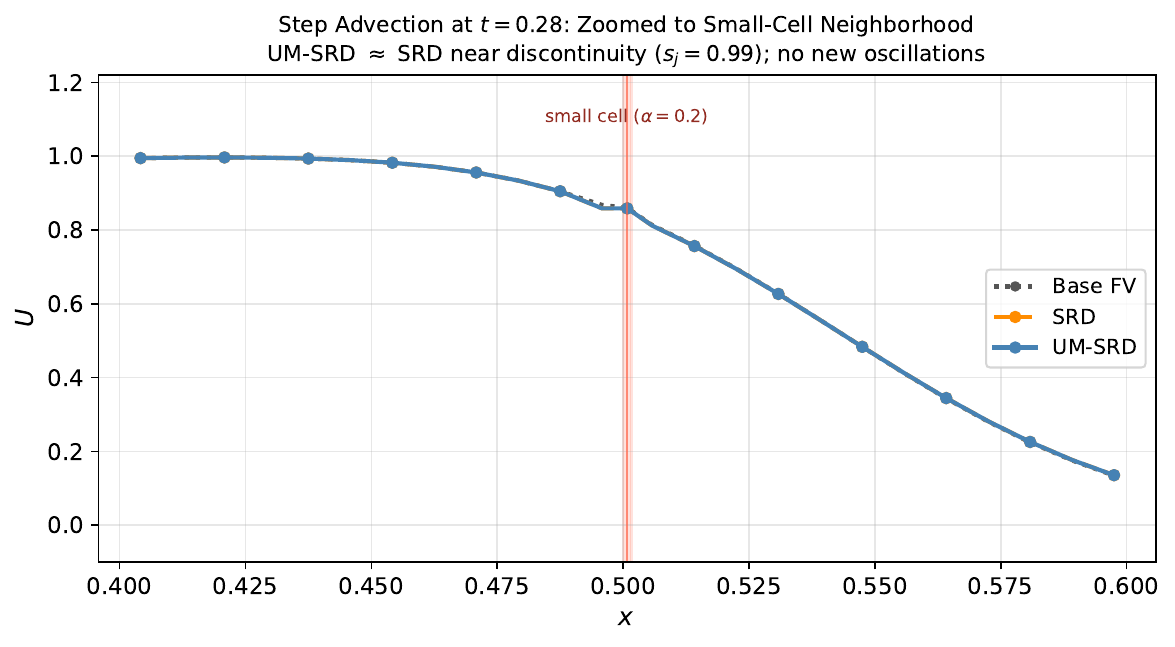}
  \caption{Experiment~3: cell-average profiles at the small-cell neighbourhood
  for the step test. All three methods are non-oscillatory; UM-SRD and SRD
  are visually indistinguishable.}
  \label{fig:step-profile}
\end{figure}

\FloatBarrier
\subsection{Experiment 4: Parameter Sensitivity}
\label{sec:exp-params}

Table~\ref{tab:param-sensitivity} reports $L^1$ errors at $N = 160$ for six 
$(p, \tau)$ combinations. Errors vary by less than 0.02\% across all tested 
values, confirming that the method is insensitive to the choice of blending 
parameters within the tested range. The default values $p = 2$ and $\tau = 
0.1$ provide a smooth transition and are used in all other experiments.

\begin{table}[ht]
  \centering
  \caption{$L^1$ errors for Experiment~4 at $N=160$.}
  \label{tab:param-sensitivity}
  \begin{tabular}{ccc}
    \hline
    $(p,\tau)$ & $L^1$ error & order \\
    \hline
    $(1,0.05)$ & 7.026e-02 & $\approx 1$ \\
    $(1,0.10)$ & 7.026e-02 & $\approx 1$ \\
    $(2,0.05)$ & 7.027e-02 & $\approx 1$ \\
    $(2,0.10)$ & 7.027e-02 & $\approx 1$ \\
    $(4,0.05)$ & 7.027e-02 & $\approx 1$ \\
    $(4,0.10)$ & 7.027e-02 & $\approx 1$ \\
    \hline
  \end{tabular}
\end{table}

\FloatBarrier
\subsection{Experiment 5: Small-Cell Instability and Stabilization}
\label{sec:exp-instability}

With $N = 100$, $\alpha = 0.05$, $a = 1$, and $\Delta t = 0.5h/a$, the 
local CFL at the small cell is 10. The base upwind scheme diverges: 
$\max|U^n|$ grows by a factor of $\approx 9$ per step, reaching $1.28 \times 
10^6$ after eight steps. Both weighted SRD and UM-SRD remain bounded with 
$\max|U| \approx 0.91$ after 200 steps, as shown in 
Figure~\ref{fig:instability}. This test isolates the stabilization mechanism: 
when the update at the small cell is large, $s_j \approx 1$ and UM-SRD 
reduces to standard SRD, inheriting its stability properties.

\begin{figure}[htbp]
  \centering
  \includegraphics[width=0.85\textwidth]{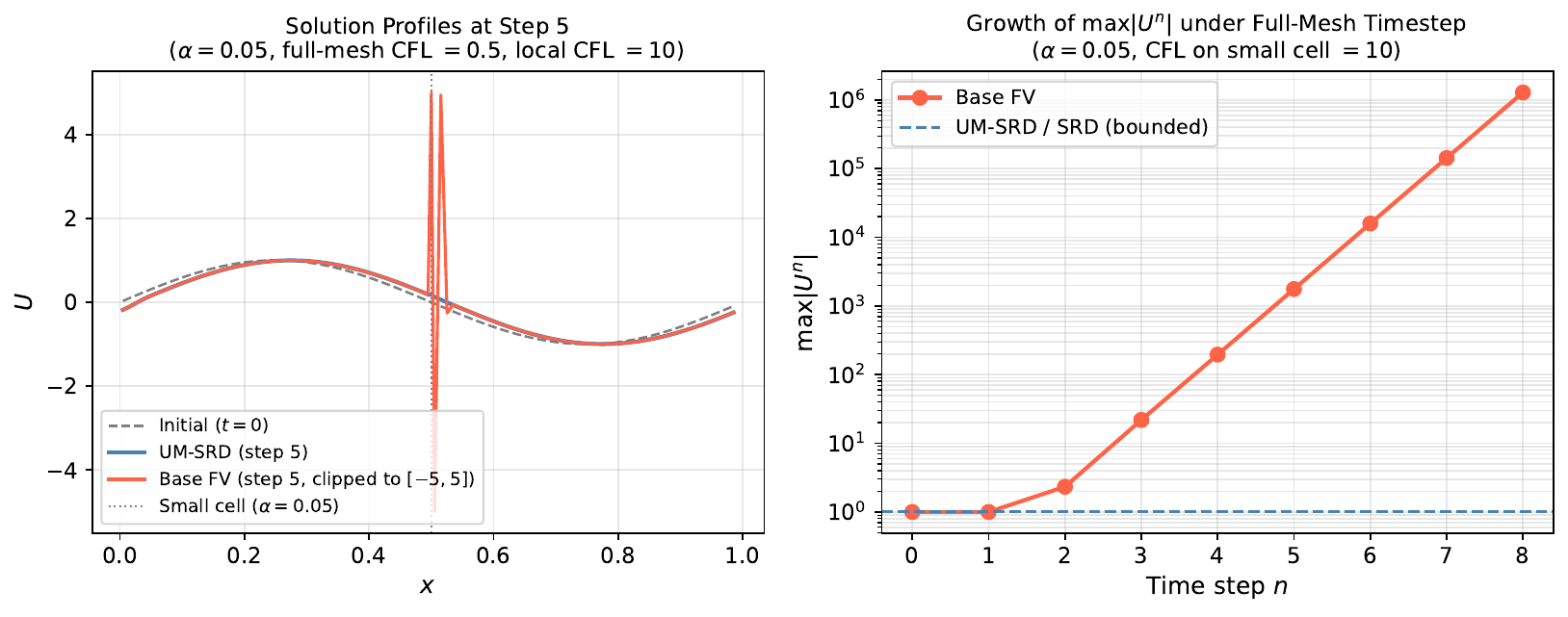}
  \caption{Experiment~5: solution profiles at step~5 (left) and growth of
  $\max|U^n|$ (right) under the full-mesh timestep with local CFL\,$=10$.
  The base scheme diverges; UM-SRD and SRD remain bounded throughout.}
  \label{fig:instability}
\end{figure}

\FloatBarrier
\subsection{Experiment 6: Steady-State Preservation}
\label{sec:exp-steady}

With $N = 50$, $\alpha = 0.2$, and a small cell at $j = 25$, we initialize 
$U_{j-1} = 1.5$, $U_j = 2.5$, and zero elsewhere, then apply redistribution 
with $U^* = U^n$ at every step, simulating a zero flux balance. After the 
first step, standard SRD sets $U_{j-1} = U_j = \hat{Q}^* = 
1.6\overline{6}$, giving $\max|U^1 - U^0| = 8.33 \times 10^{-1}$, and this 
modified profile persists at every subsequent step. UM-SRD produces 
$\max|U^k - U^0| = 0$ to machine precision at every step because $U^* = U^n$ 
forces $\eta_j = 0$ and hence $s_j = 0$ exactly. Figure~\ref{fig:steady} 
illustrates the contrast, confirming Proposition~\ref{prop:srd-steady}.

\begin{figure}[htbp]
  \centering
  \includegraphics[width=0.65\textwidth]{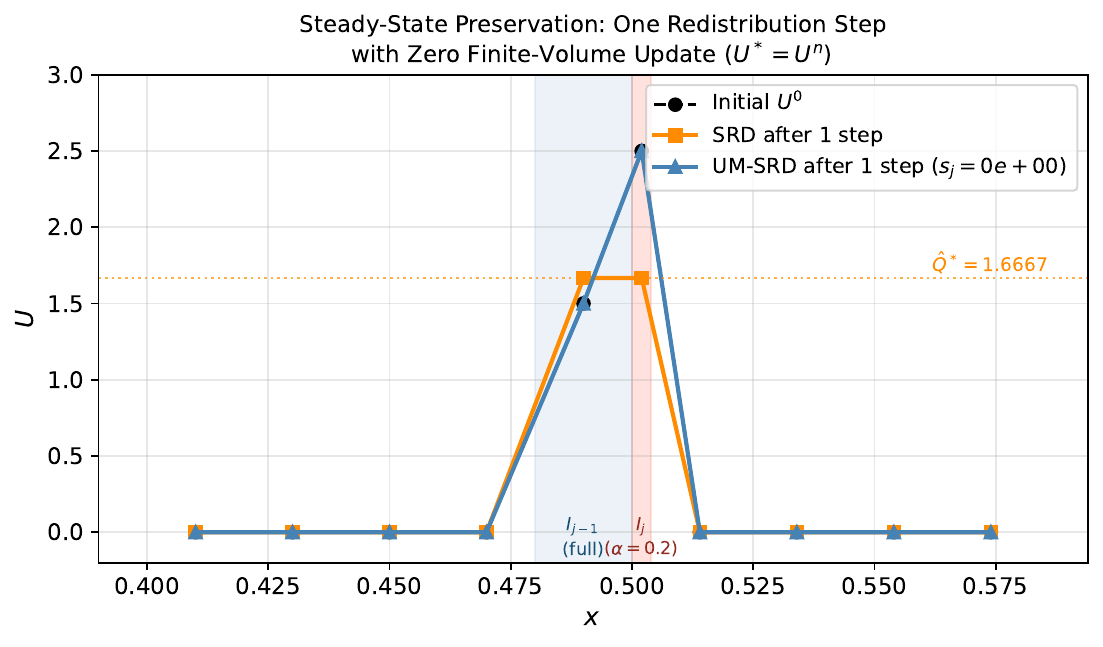}
  \caption{Experiment~6: neighbourhood profiles after one redistribution step
  with zero flux balance. SRD flattens the gradient to
  $\hat{Q}^* = 1.6\overline{6}$; UM-SRD leaves the profile unchanged
  ($s_j = 0$ exactly).}
  \label{fig:steady}
\end{figure}

\FloatBarrier
\subsection{Experiment 7: Two-Dimensional Smooth Advection}
\label{sec:exp-2d}

On the domain $[0,1]^2$ with a Cartesian grid, one cut cell at $(N/2, N/2)$ 
has its height reduced to $\alpha h = 0.2h$ by a horizontal embedded 
boundary and merges downward. The advection velocity is $(a_x, a_y) = (1, 
0.5)$, the initial condition is $u(x,y,0) = \sin(2\pi x)\sin(2\pi y)$, and 
$T = 1$. Time integration uses Godunov dimensional splitting with an $x$-sweep 
followed by a $y$-sweep at each step, with $\Delta t = 0.4h/(a_x + a_y)$. 
The UM-SRD blending operator on $M_j$ is applied after each full step and is 
identical in form to the 1D case.

Table~\ref{tab:exp7_errors} shows that UM-SRD, weighted SRD, and the base 
scheme produce nearly identical errors, agreeing to within 0.01\%, with 
observed rates approaching first order as the mesh is refined ($0.81 \to 
0.90 \to 0.95$). The pre-asymptotic sub-unity rates are consistent with the 
known behavior of first-order upwind on smooth 2D problems. 
Figure~\ref{fig:2d-convergence} shows the convergence curves and the solution 
field at $N = 160$.

\begin{table}[htbp]
\centering
\caption{Errors and convergence rates for Experiment~7 (2D, $\alpha=0.2$,
$a_x=1$, $a_y=0.5$). Errors for UM-SRD, weighted SRD, and base FV agree
to within 0.01\%.}
\label{tab:exp7_errors}
\begin{tabular}{cccccc}
\hline
$N$ & $h$ & $E^{(1)}$ & $p^{(1)}$ & $E^{(\infty)}$ & $p^{(\infty)}$ \\
\hline
40 & 0.0250 & 1.780e-01 & ---  & 4.380e-01 & --- \\
 80 & 0.0125 & 1.015e-01 & 0.81 & 2.500e-01 & 0.81 \\
160 & 0.0063 & 5.434e-02 & 0.90 & 1.340e-01 & 0.90 \\
320 & 0.0031 & 2.814e-02 & 0.95 & 6.943e-02 & 0.95 \\
\hline
\end{tabular}
\end{table}

\begin{figure}[htbp]
  \centering
  \includegraphics[width=0.92\textwidth]{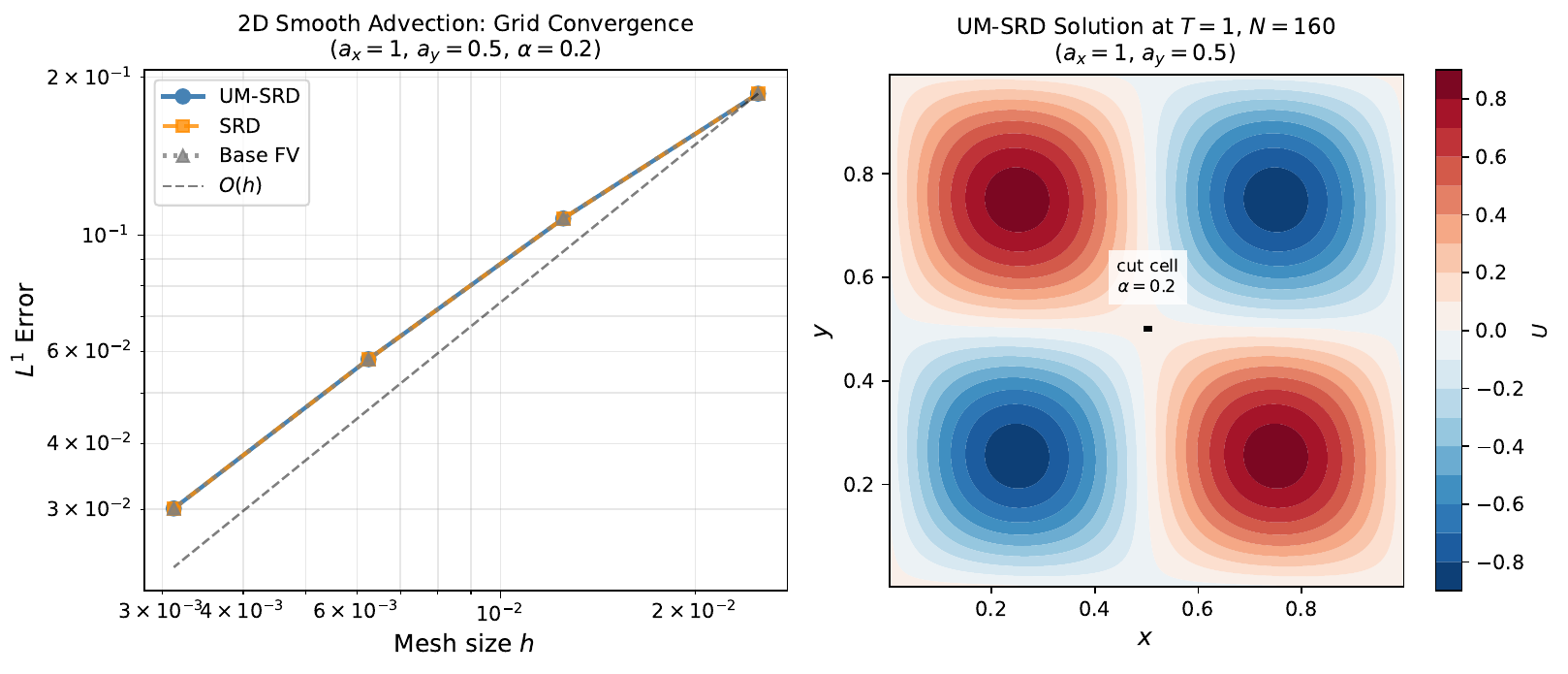}
  \caption{Experiment~7: $L^1$ convergence (left) and solution field at
  $T=1$, $N=160$ (right, cut cell outlined). All three methods are
  indistinguishable in the convergence plot.}
  \label{fig:2d-convergence}
\end{figure}

\FloatBarrier
\subsection{Experiment 8: Long-Time Near-Steady Drift}
\label{sec:exp-long-steady}

This test isolates a regime in which redistribution is not required for 
small-cell stability. With $N = 100$, $\alpha = 0.2$, and initial condition 
$U^0 = \sin(2\pi x)$, we apply redistribution with $U^* = U^n$ at every 
step, so the finite-volume update is identically zero throughout. Standard 
SRD continues to fire at every step, replacing the values in the small-cell 
neighbourhood with their volume-weighted average and accumulating a persistent 
local modification. UM-SRD produces $\|U^k - U^0\|_{L^1} = 0$ to machine 
precision at every step because $U^* = U^n$ forces $\eta_j = 0$ and hence 
$s_j = 0$ exactly.

Figure~\ref{fig:long-steady} shows the cumulative $L^1$ drift from the 
initial state as a function of redistribution step (left panel) and the 
solution profiles near the cut-cell neighbourhood at step $k = 5000$ (right 
panel). SRD flattens the local gradient toward the neighbourhood average; 
UM-SRD leaves the profile unchanged.

\begin{figure}[htbp]
  \centering
  \includegraphics[width=0.92\textwidth]{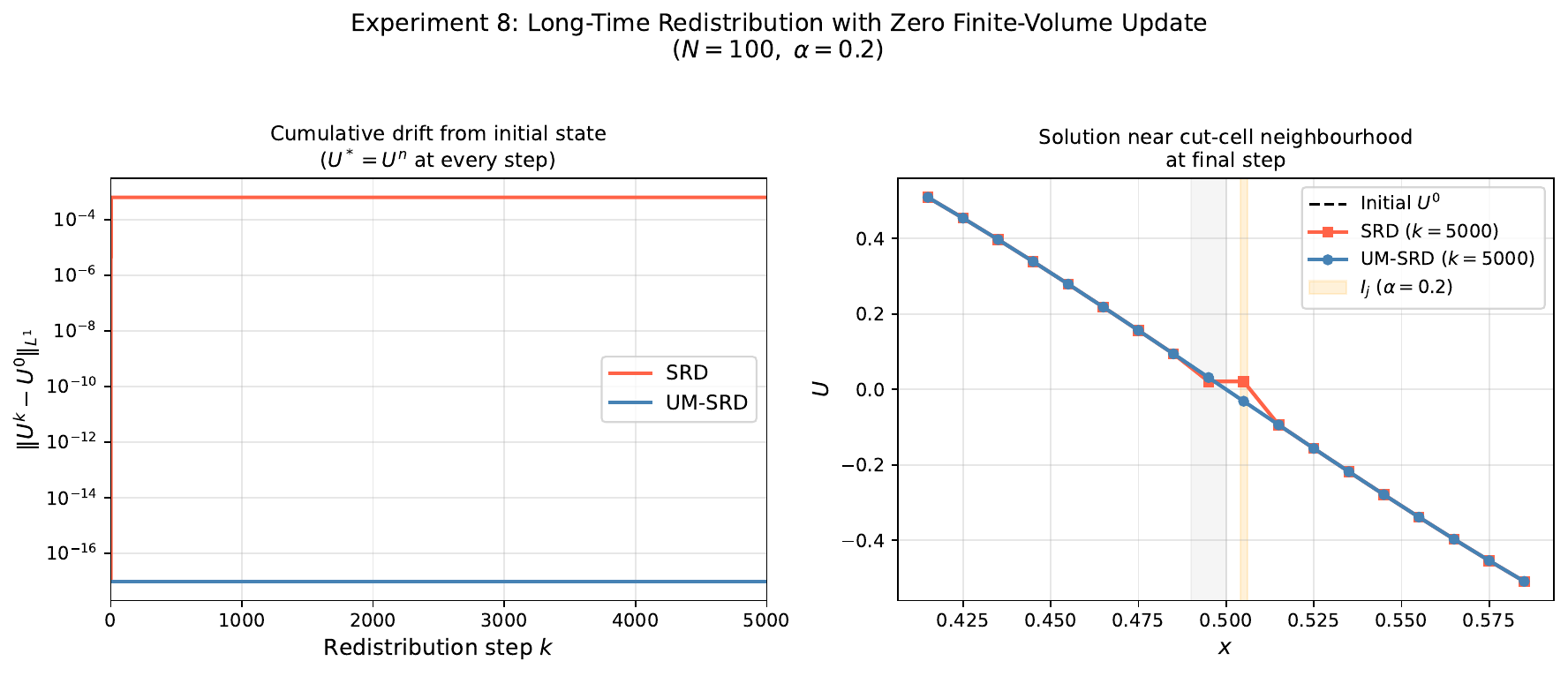}
  \caption{Experiment~8: cumulative $L^1$ drift from the initial state (left)
  and solution profiles near the cut-cell neighbourhood at $k = 5000$ (right).
  Standard SRD accumulates a persistent local modification; UM-SRD remains at
  machine precision throughout.}
  \label{fig:long-steady}
\end{figure}

\FloatBarrier
The UM-SRD blending step is dimension-independent: compute $\eta_j$, form 
$s_j$, and apply \eqref{eq:UMSRD-update}. The neighborhood topology, SRD 
weights, and flux computation are the only dimension-dependent ingredients, 
and UM-SRD inherits them unchanged from the weighted SRD implementation. 
Conservation and the shut-off property hold in any dimension by construction; 
extending the TVD proof to general multidimensional neighborhood topologies 
is left for future work.

\subsection{Experiment 9: Long-Time Advection Past a Tilted Embedded Boundary}
\label{sec:exp-long-active}

This experiment examines UM-SRD behavior over long simulation times in a 
geometry with a non-trivial embedded boundary. The domain is $[0,1]^2$ with 
a tilted embedded boundary defined by $y = 0.5 + 0.3(x - 0.5)$, producing 
96 cut cells with volume fractions ranging from 0.05 to 0.95. The advection 
velocity is $a = 1$ in $x$, the initial condition is $u(x,y,0) = 
\sin(2\pi x)\cos(2\pi y)$, and the exact solution is $u(x,y,t) = 
\sin(2\pi(x-t))\cos(2\pi y)$. We run to $T = 10$ under two CFL values: 
$0.4$ (fully active regime) and $0.005$ (near-steady regime).

In both cases, UM-SRD and SRD produce identical $L^1$ errors throughout the 
simulation, as shown in Figure~\ref{fig:exp9}. This is consistent with the 
analysis: with continuously evolving fluxes, the normalized indicator 
satisfies $s_j \approx 1$ at every cut cell at every step, and UM-SRD 
reduces to standard SRD by construction. The shut-off mechanism does not 
activate in this regime because the finite-volume update never vanishes. 
This experiment confirms that UM-SRD introduces no overhead or degradation 
in fully active advection, and that the separation between UM-SRD and SRD 
is isolated to the zero-update regime demonstrated in Experiments 6 and 8.

\begin{figure}[htbp]
  \centering
  \includegraphics[width=0.92\textwidth]{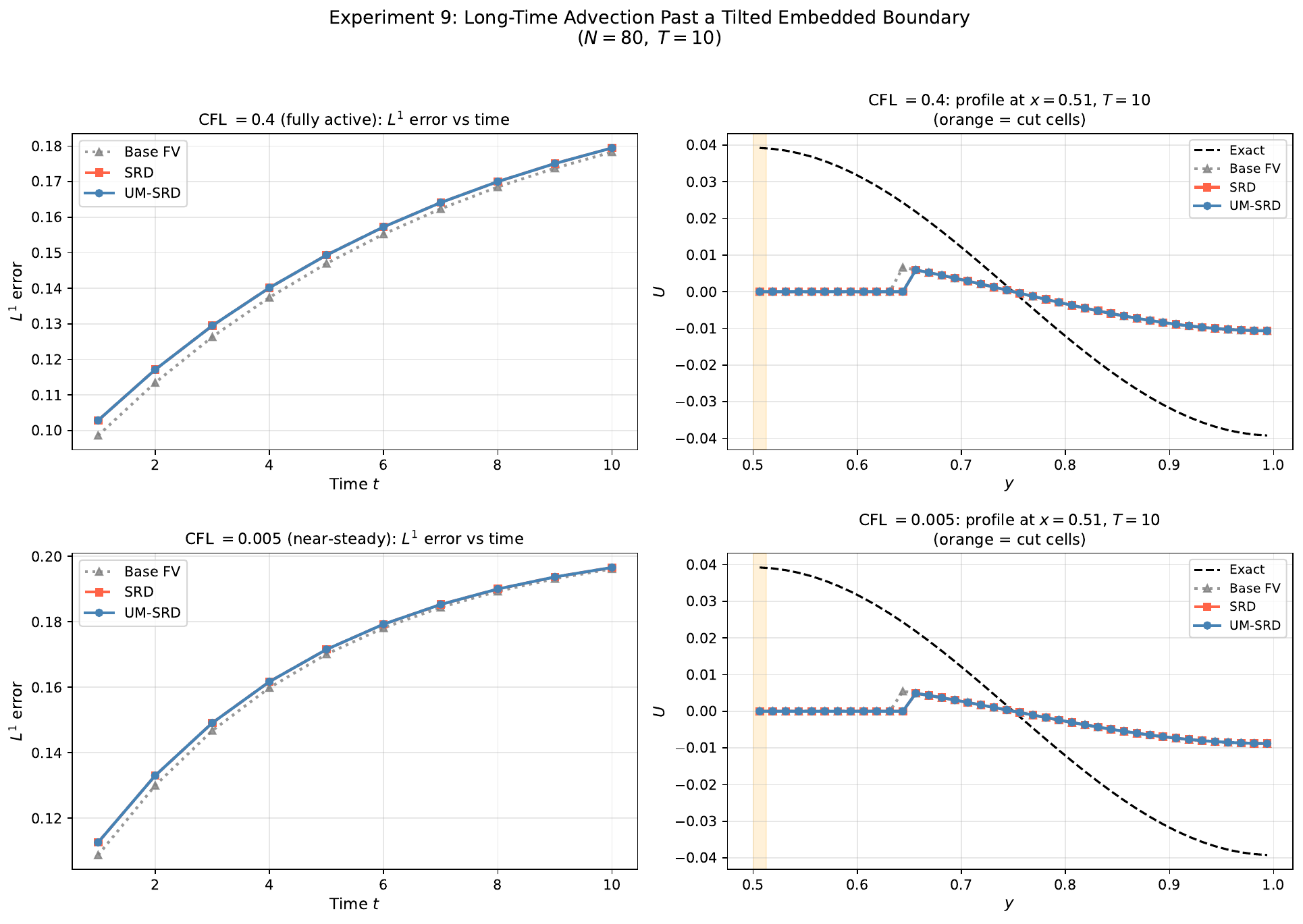}
  \caption{Experiment~9: $L^1$ error vs time (left panels) and solution
  profiles at $T=10$ (right panels) for CFL\,$=0.4$ (top) and
  CFL\,$=0.005$ (bottom). UM-SRD and SRD produce identical results in
  both regimes, consistent with $s_j \approx 1$ throughout when
  finite-volume updates are nonzero.}
  \label{fig:exp9}
\end{figure}

\FloatBarrier
\section{Algorithmic Realization}\label{sec:algorithm}

Algorithm~\ref{alg:umsrd} gives the complete UM-SRD time step. It is 
conservative by construction (Proposition~\ref{prop:conservation}), TVD on 
the 1D model problem (Theorem~\ref{thm:TVD-UMSRD}), and reduces exactly to 
the base scheme when $\Delta U_i = 0$ for all $i \in M_j$.

\begin{algorithm}
\caption{UM-SRD Update}\label{alg:umsrd}
\begin{algorithmic}[1]
\REQUIRE $U^n$, mesh, neighborhoods $\{M_j\}$, weights $\{w_{i,j}\}$,
         parameters $(p,\tau,\varepsilon)$
\ENSURE $U^{n+1}$
\STATE \textbf{Finite-volume update:} $U^*_i \leftarrow U^n_i - \frac{\Delta t}{V_i}\sum_{f}F_{i,f}$ for each $i$
\FOR{each neighborhood $j$}
  \STATE $\Delta U_{\max} \leftarrow \max_{i\in M_j}\|\Delta U_i\|_2$,\;
         $\eta_j \leftarrow \Delta U_{\max}/(\varepsilon+\Delta U_{\max})$,\;
         $s_j \leftarrow \eta_j^p/(\eta_j^p+\tau^p)$
  \STATE $\hat{V}_j \leftarrow \sum_{i\in M_j}w_{i,j}V_i$,\;
         $\hat{Q}^*_j \leftarrow \hat{V}_j^{-1}\sum_{i\in M_j}w_{i,j}V_iU^*_i$
\ENDFOR
\FOR{each cell $i$}
  \IF{$i\in M_j$ for some $j$}
    \STATE $U^{n+1}_i \leftarrow (1-s_j)U^*_i + s_j\sum_{k\in W_i}w_{i,k}\hat{Q}^*_k$
  \ELSE
    \STATE $U^{n+1}_i \leftarrow U^*_i$
  \ENDIF
\ENDFOR
\end{algorithmic}
\end{algorithm}

\section{Conclusion}

UM-SRD extends weighted SRD with a local, update-dependent blending 
parameter that reduces the redistribution operator to the identity when the 
finite-volume update vanishes. For a 1D model problem with a single small cut 
cell, the method is provably TVD under the same CFL condition as the base 
upwind scheme, conservative by construction, and exactly steady-state 
preserving. With an unnormalized indicator the perturbation to the local 
truncation error is higher-order than the base scheme; with the normalized 
indicator, which is preferred in practice for its exact machine-precision 
shut-off, first-order accuracy is preserved directly.

In fully active advection, where the finite-volume update is nonzero at 
every step, the normalized indicator satisfies $s_j \approx 1$ throughout 
and UM-SRD reduces to standard SRD. The shut-off mechanism activates only 
when updates genuinely vanish, as demonstrated by the steady-state 
preservation and long-time drift experiments. The practical benefit of 
UM-SRD relative to SRD is therefore concentrated in near-steady or 
quasi-steady flow regimes, where SRD continues to apply redistribution 
even where no stabilization is needed and UM-SRD does not.

Numerical experiments confirm these properties and illustrate two behaviors 
that distinguish UM-SRD from standard SRD: exact shut-off when the update 
vanishes, demonstrated by both the steady-state preservation test and the 
long-time drift experiment where SRD accumulates $\mathcal{O}(1)$ local 
modification while UM-SRD remains at machine precision; and full activation 
when stabilization is needed, demonstrated by the small-cell instability test 
where both methods remain bounded under a local CFL of 10 while the base 
scheme diverges in eight steps.

The analysis is restricted to a single small cell in 1D with monotone 
merging geometry. Extending the TVD and convergence theory to multiple small 
cells, general merging patterns, and multidimensional neighborhood 
topologies remains open, as does a rigorous stability proof for the 
full-mesh timestep regime. Several further properties are not addressed here 
and remain open: whether rapid temporal switching of $s_j$ between zero and 
one can introduce instabilities near marginally resolved features, whether 
UM-SRD preserves invariant domains or positivity for nonlinear systems, and 
whether partial activation near discontinuities can produce neighborhood-scale 
artifacts. These are natural directions for future work. Practically, UM-SRD 
adds only a thin layer of local blending logic to any existing weighted SRD 
implementation and inherits its neighborhood and weight structure unchanged.
\section*{Acknowledgments}

The author thanks Marsha Berger and Andrew Giuliani for correspondence
regarding the stability mechanism of the blending operator and for
helpful comments on an early version of this work. The author used Claude (Anthropic) for assistance with manuscript 
editing, LaTeX formatting, and Python code for the numerical experiments.


\end{document}